\documentclass[12pt]{article}

\usepackage{amsfonts}
\usepackage{amssymb}
\usepackage{amsthm}
\usepackage{amsmath}
\usepackage{mathrsfs}
\usepackage[latin1]{inputenc}
\usepackage[all]{xy}
\usepackage[left=1.20in,right=1.20in]{geometry}
\usepackage{xcolor}

\newcommand{\rac}{\mathbb Q}

\newcommand{\bind}{\textrm{Ind}}
\newcommand{\bres}{\textrm{Res}}
\newcommand{\bdef}{\textrm{Def}}
\newcommand{\binf}{\textrm{Inf}}

\newcommand{\bten}{\textrm{Ten}}

\theoremstyle{plain}
\newtheorem{teo}{Theorem}[section]
\newtheorem{prop}[teo]{Proposition}
\newtheorem{coro}[teo]{Corollary}
\newtheorem{lema}[teo]{Lemma}

\theoremstyle{definition}

\theoremstyle{remark}
\newtheorem{rem}[teo]{Remark}
\newtheorem{ejem}[teo]{Example}

\title{Deflation and tensor induction on the Frobenius-Wielandt morphism}
\author{Nadia Romero\footnote{\texttt{nadia.romero@ugto.mx}}\\ 
\begin{small}
Departamento de Matem\'aticas,
\end{small}\\
\begin{small}
Universidad de Guanajuato, Mexico.
\end{small}
}

\date{ }

\begin{document}

\maketitle

\begin{abstract}
We explore conditions for the Frobenius-Wielandt morphism to commute with the operations of deflation and tensor induction on the Burnside ring. In doing this, we review the commutativity with induction. The techniques used for induction and tensor induction no longer work for deflation, so in this case we make use of tools coming from the theory of biset functors.
\end{abstract}

\section*{Introduction}

The Frobenius-Wielandt morphism was introduced by Dress, Siebeneicher and Yoshida in \cite{lostres}. For a finite group $G$, it is a ring homomorphism $\alpha^G$ from the Burnside ring $B(C)$, of a cyclic group $C$ of order $|G|$, to the Burnside ring $B(G)$ of $G$. The main property of $\alpha^G$ is that for any element $x$ in $B(C)$ and any subgroup $K$ of $G$, the number of fixed points of $\alpha^G(x)$ by $K$ equals the number of fixed points of $x$ by ${C_{|K|}}$, where $C_{|K|}$ is the subgroup of $C$ of order $|K|$. According to the authors (see Remark 2 in \cite{lostres}), this gives a precise conceptual interpretation of
the observation that many elementary group-theoretic results can be derived from the fact that various invariants of an arbitrary group are closely related to the same invariant evaluated for
the cyclic group $C$.

In Section 5 of \cite{lostres} the functorial properties of $\alpha^G$ are considered. It is shown that $\alpha^G$ commutes with restriction, $\bres^G_H:B(G)\rightarrow B(H)$, for every subgroup $H$ of $G$, and that it commutes with the operation of fixed points $\textrm{Fix}^G_{G/N}:B(G)\rightarrow B(G/N)$ for every normal subgroup $N$ of $G$. On the other hand, it is also shown that $\alpha^G$ commutes with inflation $\binf^{\,G}_{G/N}: B(G/N)\rightarrow B(G)$ for $N$ a normal subgroup of $G$, if and only if $G$ and $N$ satisfy that for any subgroup $K$ of $G$, we have $|K\cap N|=g.c.d.(|K|,\, |N|)$. As we will see, this property on $G$ and $N$ will be very important, we will call it \textit{the gcd property}. We will show that the Frobenius-Wielandt morphism commutes with the operation of tensor induction $\bten_H^G:B(H)\rightarrow B(G)$ if and only if $G$ and $H$ have the gcd property, and that the same condition applies for the induction $\bind_H^G: B(H)\rightarrow B(G)$. The result for induction can be obtained from the results appearing in Section 5 of \cite{lostres} too, although the commutativity is not considered in this way there.

As usual, the operation of deflation $\bdef_{G/N}^{\,G}:B(G)\rightarrow B(G/N)$, for $N$ a normal subgroup of $G$, is much more complicated. An example of the fact that the Frobenius-Wielandt morphism does not commute in general with deflation is given in \cite{miguel}. So, we searched for a sufficient and necessary condition for this commutativity. Nevertheless, the methods used to do this on induction and tensor induction do not work for deflation, mainly because there is no formula for the number of fixed points of a deflated $G$-set. In spite of this, we  prove that if the Frobenius-Wielandt morphism commutes with $\bdef_{G/N}^{\,G}$, then $G$ and $N$ have the gcd property and $N$ is a cyclic central subgroup of $G$. We do this by means of the constants $m_{G,\, N}$, introduced by Serge Bouc in \cite{foncteurs} to study the Burnside biset functor. We finish this note by giving some examples for which the Frobenius-Wielandt morphism commutes with deflation under certain conditions.

\section{Preliminaries}

We begin by introducing some notation. In what follows $G$ is a finite group. 

The conjugate $aga^{-1}$ of an element $g\in G$ by $a\in G$ will be denoted by $^ag$, and $g^a$ will stand for $a^{-1}ga$. We also write $^a\!K$ and $K^a$ for the conjugates of a subgroup $K$ of $G$. A set of representatives of the conjugacy classes of subgroups of $G$ will be denoted by $[s_G]$.

The Frattini subgroup of a group $G$ is denoted by $\Phi (G)$, and we denote by $\Delta (G)$ the group $\{(a,\, a)\mid a\in G\}$. 

The greatest common divisor of two integers $n$ and $m$ will be denoted simply by $(n,\, m)$. If $n$ is an integer and $p$ is a prime number, we denote by $n_p$ the highest power $p^k$ dividing $n$.


Groups satisfying any of the conditions  of the following lemma will be important in the next sections, so we prove a couple of results about them.

\begin{lema}
\label{propiedad}
Let $N$ be a subgroup of $G$. The following are equivalent:
\begin{itemize}
\item[i)] For every $H\leqslant G$ we have $|H\cap N|=(|H|,\, |N|)$.
\item[ii)] If  $H\leqslant G$ is such that $|H|$ divides $|N|$, then $H\subseteq N$.
\item[iii)] If  $H\leqslant G$ is a cyclic group such that $|H|$ divides $|N|$, then $H\subseteq N$.
\item[iv)] For every cyclic subgroup $H$ of $G$ we have $|H \cap N|=(|H|,\, |N|)$.
\end{itemize}
\end{lema}
\begin{proof}
Clearly $iv)$ implies $iii)$. 
Suppose $iii)$ and let $H$ be a subgroup of $G$ of order dividing $|N|$. If $x\in H$, then $<x>$ is a cyclic subgroup of $G$ of order dividing $|N|$, thus $x\in N$. Hence we have $iii)\Rightarrow ii)$.

Now suppose $ii)$ and let $p$ be a prime that divides $|N|$ and $|H|$. Let $p^r$ be  $(|H|,\, |N|)_p$. Then $H$ has a subgroup $H_1$ of order $p^r$, since its $p$-Sylow subgroup does, but by $ii)$ we have $H_1\subseteq N$, hence $H_1\subseteq H\cap N$. This implies that $p^r$ divides $|H\cap N|$ and then $|H\cap N|=(|H|,\, |N|)$. Thus we have $ii)\Rightarrow i)$ and clearly $i)\Rightarrow iv)$.
\end{proof}

%

As said in the Introduction, we will refer to the property of the previous lemma as  \textit{the gcd property}. Observe that if $N\leqslant G$  satisfy the gcd property, then $N$ must be normal in $G$. A characterization of the groups $G$ and $N$ satisfying this property is given in \cite{lostres} without proof, for the sake of completeness, we prove it next. Notice that this characterization shows that the gcd property depends on the couple $G$ and $N$. Nevertheless, for simplicity sometimes we will just say that $N$ has the gcd property.

\begin{lema}
\label{syss}
Let $N$ be a normal subgroup of $G$. Then $G$ and $N$ have the gcd property if and only if for each prime $p$ diving $|G|$ we have one of the following three possibilities: 
\begin{itemize}
\item[i)] $|N|_p=1$.
\item[ii)] $|N|_p=|G|_p$.
\item[iii)] The Sylow $p$-subgroups of $G$ are cyclic or generalized quaternion groups, and if they are generalized quaternion groups, then $|N|_2=2$.
\end{itemize}
\end{lema}
\begin{proof}
First we prove the \textit{if} part.  
To prove  $|H\cap N|=(|H|,\, |N|)$ for $H\leqslant G$, it suffices to prove that for each prime $p$ we have that $(|H|,\, |N|)_p$ divides $|H\cap N|$. To do this, we let $H_1$ be a subgroup of $H$ or order $p^r=(|H|,\, |N|)_p$ and we show that in each of the three cases of the statement, $H_1$ is contained in $H\cap N$. 
If $|N|_p=1$ then this is clear.  Next suppose $|N|_p=|G|_p$. Then, every Sylow $p$-subgroup of $G$ is contained in $N$ and so $H_1$ is contained in $H\cap N$. Finally, suppose we have $iii)$. Let  $P$ be a Sylow $p$-subgroup of $G$, containing $H_1$. Since $P\cap N$ is a Sylow $p$-subgroup of $N$, we have $|P\cap N|=|N|_p$, and so $p^r$ divides $|P\cap N|$, thus $P\cap N$ has  a subgroup of order $p^r$. If $P$ is cyclic, then $H_1$ is its unique subgroup of order $p^r$ and hence it is  contained in $H\cap N$. If $P$ is a generalized quaternion and $|N|_2=2$, then $p^r$ is either 1 or 2, in either case $H_1$ is contained in $H\cap N$, since a generalized quaternion group has a unique subgroup of order 2.

Now suppose that $N\leqslant G$ satisfy the gcd property, and let $p$ be a prime dividing $|G|$ with $|G|_p=p^n$. If neither $i)$ nor $ii)$ are satisfied, then $|N|_p=p^k$ with $0<k<n$. Let $P$ be a Sylow $p$-subgroup of $G$ and $P_k=P\cap N$. Observe that, by point $ii)$ of the previous lemma, $P_k$ is the only subgroup of $P$ of order $p^k$.
We will show that this is all what is needed for $P$ to be of the type described in $iii)$. That is, we show that if $Q$ is a finite $p$-group (of order $p^n$, as before) containing a unique non-trivial, proper subgroup, of order $p^k$, then $Q$ is either cyclic or generalized quaternion. We show this by induction on $n-k$. If $n=k+1$, then $Q$ contains a unique maximal subgroup, thus $Q$ must be cyclic. Suppose now $n-k>1$, then, by induction, every maximal subgroup of $Q$ is either cyclic or generalized quaternion. If $Q$ is abelian, then  it must be cyclic. If $Q$ is not abelian, then any abelian subgroup of $Q$ is contained in a maximal subgroup and hence is cyclic, thus, by Theorem 4.10 in \cite{gore} $Q$ is cyclic or generalized quaternion. Also, if $p=2$ and $Q$ is a generalized quaternion, then $p^k$ must be equal to $2$ since only in this case a generalized quaternion group has a unique non-trivial proper subgroup.
\end{proof}

\subsection{The Burnside ring}
\label{pre-burn}

The Burnside ring of a finite group will be denoted by $B(G)$, and by $\rac B(G)$ if we are taking coefficients in $\rac$. For a subgroup $U\leqslant G$, the ring homomorphism of fixed points $B(G)\rightarrow \mathbb{Z}$ will be denoted by $x\mapsto |x^U|$.

The following is Theorem 1 in \cite{lostres}, it describes the main property of the Frobenius-Wielandt morphism. The group $C$ is a cyclic group of order $|G|$.
 
\begin{teo}
\label{dsy}
There exists a ring homomorphism
\begin{displaymath}
\alpha^G:B(C)\rightarrow B(G),
\end{displaymath}
which we call the Frobenius-Wielandt homomorphism from the Burnside ring $B(C)$ of the cyclic group $C$ into the Burnside ring $B(G)$ of $G$ such that for every subgroup $U\leqslant G$ of $G$ and every $x\in B(C)$, the number $|\alpha (x)^U|$ of $U$-invariant elements in the virtual $G$-set $\alpha (x)$ coincides with the corresponding number $|x^{C_{|U|}}|$ of $C_{|U|}$-invariant elements in $x$, where, of course,
$C_{|U|}$ denotes the unique subgroup of order $|U|$ in $C$.
\end{teo}

When referring to the Frobenius-Wielandt morphism in general, we will simply write FW morphism. The construction of this morphism in \cite{lostres} depends on the use of a particular basis of $B(C)$, that is, the usual basis of transitive $C$-sets is not used to define the FW morphism and, in general, one does not know the value of $\alpha^G$ at an element of the form $[C/D]$ for $D\leqslant C$. The following easy lemma explores this question.

\begin{lema}
\label{trans}
Let $G$ be a finite group, $C$ and $\alpha^G$ as in the previous theorem. For $D$ a subgroup of $C$ we have that $\alpha^G([C/D])=[G/N]$ if and only if $N$ is a subgroup of $G$ of order $|D|$ that satisfies the gcd property.
\end{lema}
\begin{proof}
Suppose $\alpha^G([C/D])=[G/N]$ for some $N\leqslant G$. By taking the points fixed by the trivial group $1$, we have $|(G/N)^1|=[G:N]$, and  by the previous theorem
\begin{displaymath}
|\alpha^G([C/D])^1| =|(C/D)^1|=|C/D|,
\end{displaymath}
hence $N$ has order $|D|$. Now, by taking the points fixed by $N$ we have $|(G/N)^N|=[N_G(N):N]$, and on the other hand
\begin{displaymath}
\left|\left(\alpha^G([C/D])\right)^N\right| =\left| \left(C/D\right) ^{C_{|N|}}\right|=|C/D|.
\end{displaymath}
Hence $N$ is a normal subgroup of $G$. Finally, letting $H\leqslant G$ be such that $|H|$ divides $|N|$, we have as before $|(\alpha^G([C/D]))^H|=|C/D|$. So, in order to have $|(G/N)^H|=|C/D|$ we must have that $H$ is contained in $N$, and so $N$ satisfies the gcd property.

Now suppose that $N$ is a subgroup of order $|D|$ that satisfies the gcd property. Then $N$ is a normal subgroup of $G$ and for $K\leqslant G$, by Lemma \ref{propiedad}, $\left| \left(G/N\right)^K\right|$ is equal to $|G/N|=|C/D|$ if $|K|$ divides $|N|=|C|$ and zero otherwise.  But by Theorem \ref{dsy}, the same holds for $\left|\left(\alpha^G([C/D])\right)^K\right|$, hence $\alpha^G([C/D])=[G/N]$.
\end{proof}

\begin{rem}
\label{remtrans} 
This lemma gives another characterization of groups satisfying the gcd property: $G$ and $N$ satisfy the gcd property if and only if for every subgroup $K$ of $G$ we have
\begin{displaymath}
|(C/C_N)^{C_{|K|}}|=|(G/N)^K|.
\end{displaymath}
\end{rem}

\subsubsection*{Deflation and tensor induction on the Burnside ring}
\label{pre-deften}

We recall that given $G$ and $H$ groups, a $(G,\, H)$\textit{-biset} is a set $X$ with a left $G$-action and a right $H$-action such that $(gx)h=g(xh)$ for all $g\in G$, $x\in X$ and $h\in H$. It is well known that given a  $(G,\, H)$-biset $X$, one can define an additive morphism $B(X):B(H)\rightarrow B(G)$. The basic operations of induction, restriction, inflation and deflation of the Burnside ring can be recovered in this way, through appropriate bisets. This is part of the foundations of the theory of \textit{biset functors} (see for example Chapter 2 of \cite{boucbook} for more details). But, in Section 11.2 of \cite{boucbook} it is shown that given the biset $X$, one can also define a function $T_X:B(H)\rightarrow B(G)$ which is multiplicative, that is, such that $T_X(ab)=T_X(a)T_X(b)$ for $a,\,b\in B(H)$. The basic operations one obtains from this construction are the usual inflation and restriction (they are known to be multiplicative), but also the tensor induction and the fixed points operation. It is shown in Proposition 11.2.20 of \cite{boucbook} that this construction is functorial with respect to bisets.


Remember that for $N\trianglelefteq G$, the deflation of the Burnside ring is an additive morphism
\begin{displaymath}
\bdef_{G/N}^{\,G}:B(G)\rightarrow B(G/N),
\end{displaymath}
that sends a $G$-set $X$ to the set of equivalence classes $X/\sim$, under the relation $x\sim nx$, for $n\in N$ and $x\in X$. This set has a natural structure of $G/N$-set. On the other hand, if $H\leqslant G$, the tensor induction is a multiplicative map
\begin{displaymath}
\bten_H^G:B(H)\rightarrow B(G),
\end{displaymath}
built by means of Lemma 11.2.15 of \cite{boucbook}, that sends an $H$-set $X$ to Hom$_H(G^{op},\, X)$. This set consists of the maps $f:G\rightarrow X$ such that $f(gh)=h^{-1}f(g)$ for any $g\in G$ and $h\in H$. The group $G$ acts on this set by $gf(g_1)=f(g^{-1}g_1)$. 

We mention some properties of these operations we will use later on.

The following lemma is a particular case of Corollary 11.2.17 of \cite{boucbook}.
\begin{lema}
\label{ptosdef}
Let $G$ be a finite group and $H,\, K\leqslant G$. For $a\in B(H)$ we have
\begin{displaymath}
\left|\emph{Ten}_H^G(a)^K\right|=\prod_{g\in [K\backslash G/H]}\left|a^{K^g\cap H}\right|,
\end{displaymath}
where $[K\backslash G/H]$ is  a set of representatives of $K\backslash G/H$.
\end{lema} 

Regarding deflation, we will use the following results.

Recall that $\rac B(G)$ is a split semi-simple commutative $\rac$-algebra, its primitive idempotents are indexed, up to conjugation, by the subgroups of $G$. More precisely we have:

\begin{teo}[Theorem 2.5.2 in \cite{boucbook}]
Let $G$ be a finite group. If $H$ is a subgroup of $G$, denote by $e_H^G$ the element of $\rac B(G)$ defined by
\begin{displaymath}
e_H^G=\frac{1}{|N_G(H)|}\sum_{K\leqslant H}|K|\mu(K,\, H)[G:K]
\end{displaymath}
where $\mu$ is the M\"obius function of the poset of subgroups of $G$. Then $e_H^G=e_K^G$ if the subgroups $H$ and $K$ are conjugate in $G$, and the elements $e_H^G$, for $H\in[s_G]$, are the primitive idempotents of the $\rac$-algebra $\rac B(G)$. 
\end{teo}
 
We refer to \cite{boucbook} for this theorem but of course, this result can be found in other sources. The action of bisets on these idempotents is described in Theorem 5.2.4 of \cite{boucbook}. Since we will only be interested in the action of deflation, we recall this particular case of the theorem.

\begin{lema}\label{ladef}
Le $H$ be a subgroup of $G$ and $N\trianglelefteq G$. Considering the idempotent $e_H^G$ as before, we have
\begin{displaymath}
\emph{Def}_{G/N}^{\,G}e_H^G=\frac{|N_G(HN)/HN|}{|N_G(H)/H|}m_{H,\, H\cap N}e_{HN/N}^{G/N}.
\end{displaymath}
The constant $m_{H,\, H\cap N}$ is defined in general for groups $K\trianglelefteq L$ as
\begin{displaymath}
m_{L,\, K}=\frac{1}{|L|}\sum_{XK=L}|X|\mu(X,\,L),
\end{displaymath}
where $\mu$ is the M\"obius function of the poset of subgroups of $L$.
\end{lema}

We finish the section by mentioning some properties of the constants $m_{L,\,K}$.

\begin{lema}
\label{propiedadesm}
Let $G$ be a finite group.
\begin{itemize}
\item[i)] Suppose that $N$ and $M$ are normal subgroups of $G$ with $N\leqslant M$, then
\begin{displaymath}
m_{G,\, M}=m_{G,\,N}m_{G/N,\, M/N}.
\end{displaymath}
\item[ii)] The constant $m_{G,\, G}$ is different from zero if and only if $G$ is cyclic, and if $G$ is cyclic of order $n$, then $m_{G,\, G}=\varphi(n)/n$, where $\varphi$ is the Euler totient function.
\item[iii)] If $G$ is cyclic of order $n$ and $N$ is a subgroup of order $m$, then 
\begin{displaymath}
m_{G,\, N}=\frac{\varphi(n)}{m\varphi(n/m)}.
\end{displaymath}
\end{itemize}
\end{lema}
\begin{proof}
Points $i)$ and $ii)$ are propositions 5.3.1 and 5.6.1 in \cite{boucbook}, respectively. Point $iii)$ is easily obtained from $i)$ and $ii)$.
\end{proof}


\section{Functorial properties of the FW morphism}

Throughout this section $G$ denotes a finite group, $C$ denotes a cyclic group of order $|G|$ and if $H$ is a subgroup of $G$, for simplicity, we will denote by $C_H$ the subgroup of $C$ of order $|H|$. The FW morphism is denoted, as before, by $\alpha^G:B(C)\rightarrow B(G)$.

As said in the Introduction, regarding the operations of restriction and fixed points, it is shown in Section 5 of \cite{lostres} that the FW morphism commutes with both of them. It is also shown that the FW morphism commutes with inflation, $\binf_{G/N}^{\,G}$, if and only if $N$ and $G$ satisfy the gcd property. For induction and tensor induction we have analogous results.

\subsection{Induction and tensor induction}

\begin{lema}
Let $H$ be a subgroup $G$ and consider the induction morphism $\emph{Ind}_H^{\,G}:B(H)\rightarrow B(G)$. The following diagram commutes
\[
\xymatrix{
B(C_H)\ar[r]^{\alpha^H}\ar[d]_{\emph{Ind}_{C_H}^{\,C}}&B(H)\ar[d]^{\emph{Ind}_H^{\,G}}\\
B(C)\ar[r]_{\alpha^G}&B(G)
}
\]
if and only if $G$ and $H$ satisfy the gcd property. In particular $H$ must be a normal subgroup of $G$.
\end{lema}
\begin{proof}
Let $K$ be a subgroup of $G$. Recall, for example by  Proposition 2.2.1 of \cite{burnamscorr}, that for $a\in B(H)$ we have
\begin{displaymath}
|\bind_H^G(a)^K|=\sum_{\substack{g\in N_G(K)\backslash G/H\\K^g\subseteq H}}\left|\left(\bind_{N_{^gH}(K)/K}^{N_G(K)/K}({}^ga)\right)^K\right|,
\end{displaymath}
hence if $K$ is not contained in a conjugate of $H$, then $|\bind_H^G(a)^K|=0$. So, for $x$ in $B(C_H)$, observing that $^g(\alpha^H)=\alpha^{^gH}$, we have  by Theorem \ref{dsy} 
\begin{displaymath}
\left|\left(\bind_H^G\alpha^H(x)\right)^K\right|= |(G/H)^K||x^{C_K}|
\end{displaymath}
and
\begin{displaymath}
\left|\left(\alpha^G\bind_{C_H}^{\,C}(x)\right)^K\right|= |(C/C_H)^{C_K}||x^{C_K}|.
\end{displaymath}
If $G$ and $H$ satisfy the gcd property, then by Remark \ref{remtrans}, $|(C/C_H)^{C_K}|=|(G/H)^K|$ for every $K\leqslant G$, and we conclude that the diagram commutes. Now suppose that the diagram commutes and let $K\leqslant G$ be such that $|K|$ divides $|H|$.  Taking $x=[C_H/C_H]$ in the previous equalities, we have that $|(C/C_H)^{C_K}|=|(G/H)^K|$. But then $|(C/C_H)^{C_K}|=|(G/H)^K|$ for every $K\leqslant G$ and, again by Remark \ref{remtrans}, $G$ and $H$ satisfy the gcd property.
\end{proof}

The previous lemma can also be obtained, using Lemma \ref{trans}, from the results appearing in Section 5 of \cite{lostres}. Next we prove the corresponding result for tensor induction.

\begin{prop}
Let $H$ be a subgroup of $G$. 
The following diagram commutes
\[
\xymatrix{
B(C_H)\ar[r]^{\alpha^H}\ar[d]_{\emph{Ten}_{C_H}^C}&B(H)\ar[d]^{\emph{Ten}_H^G}\\
B(C)\ar[r]_{\alpha^G}&B(G)
}
\]
if and only if $G$ and $H$ satisfy the gcd property. In particular $H$ must be a normal subgroup of $G$.
\end{prop}
\begin{proof}
Suppose $\alpha^G\bten_{C_H}^C=\bten_H^G\alpha^H$. Then for any $x\in B(C)$ and any $K\leqslant G$ we must have
\begin{displaymath}
\left| \left(\bten_H^G\alpha^H(x)\right)^K \right| =\left| \left(\alpha^G\bten_{C_H}^C(x)\right)^K\right|.
\end{displaymath}
By Theorem \ref{dsy}, the right-hand side of this equality is equal to $\left| \left(\bten_{C_H}^C(x)\right)^{C_K}\right|$ and so by Lemma \ref{ptosdef}, we have
\begin{displaymath}
\prod_{g\in[K\backslash G/H]}\left|\alpha^H(x)^{K^g\cap H}\right|=\prod_{t\in[C_K\backslash C/C_H]}\left|x^{C_K^t\cap C_H}\right|,
\end{displaymath}
which, again by Theorem \ref{dsy}, is equal to
\begin{displaymath}
\prod_{g\in[K\backslash G/H]}\left|x^{C_{K^g\cap H}}\right|=\prod_{t\in[C/C_KC_H]}\left|x^{C_K\cap C_H}\right|.
\end{displaymath}
Since this must hold for every $x$, let $n>1$ be an integer and take $x=n[C/C]$ in $B(C)$. Then the equality above gives
$n^{|K\backslash G/H|}=n^{|C/C_KC_H|}$, which implies
\begin{displaymath}
|K\backslash G/H|=\frac{|G|(|K|, \, |H|)}{|K||H|}.
\end{displaymath}
Finally, since 
\begin{displaymath}
\frac{|G|}{|K||H|}=\sum_{g\in [K\backslash G/H]}\frac{1}{|K^g\cap H|}
\end{displaymath}
and $|K^g\cap H|\leq (|K|,\, |H|)$ for every $g\in G$, for the equality above to hold we must have $|K^g\cap H|=(|K|,\, |H|)$ for every $g\in [K\backslash G/H]$. In particular $|K\cap H|=(|K|,\, |H|)$ for any $K\leqslant G$. This proves the \textit{only if} part of the statement. 

Now suppose that $|K\cap H|=(|K|,\, |H|)$ for every $K\leqslant G$. Since $H$ is a normal subgroup of $G$, we have $K\backslash G/H=G/KH$. Also $|KH|=|C_KC_H|$ and $|K^g\cap H|=|C_K\cap C_H|$ for any $g\in G$. This means that for any $x\in B(C)$ we have
\begin{displaymath}
\prod_{g\in[K\backslash G/H]}\left|x^{C_{K^g\cap H}}\right|=\prod_{t\in[C/C_KC_H]}\left|x^{C_K\cap C_H}\right|.
\end{displaymath}
Hence $\alpha^G\bten_{C_H}^C=\bten_H^G\alpha^H$.
\end{proof}

\subsection{Deflation}

In this section $G$ denotes a finite group and $N$ a normal subgroup of $G$. The rest of the notation remains the same. Since $C/C_N$ is isomorphic to $C_{G/N}$, we will identify $B(C/C_N)$ with $B(C_{G/N})$.

It is easy to see that the following diagram commutes
\[
\xymatrix{
B(C)\ar[r]^{\alpha^G}\ar[d]_{\bdef_{C/C_N}^{\,C}}&B(G)\ar[d]^{\bdef_{G/N}^{\,G}}\\
B(C_{G/N})\ar[r]_{\alpha^{G/N}}&B(G/N)
}
\]
if and only if the corresponding diagram with coefficients in $\mathbb{Q}$ commutes. The reason for taking coefficients in $\mathbb{Q}$ is that in $\mathbb{Q}B(G)$ we can consider the family of primitive idempotents $e_H^G$, for $H\in[s_G]$, defined in Section \ref{pre-burn}. So, in what follows, we will work with $\mathbb{Q}B(C)$ and $\mathbb{Q}B(G)$. For simplicity, we will continue writing the morphisms as $\alpha^G$ and $\bdef_{G/N}^{\,G}$. 

We begin with the following property.

\begin{prop}
\label{defsub}
If for $G$ and $N$ the FW morphism commutes with $\emph{Def}_{G/N}^{\,G}$, then for every subgroup $T$ of $G$ containing $N$, the FW morphism commutes with $\emph{Def}_{T/N}^{\,T}$.
\end{prop}
\begin{proof}
Consider the following diagram
\[
\xymatrix{
\rac B(C)\ar[r]^{\alpha^G}\ar[d]_{\bres_{C_T}^{C}}&\rac B(G)\ar[d]^{\bres_{T}^{G}}\\
\rac B(C_T)\ar[r]^{\alpha^{T}}\ar[d]_{\bdef_{C_T/C_N}^{\,C_T}}&\rac B(T)\ar[d]^{\bdef_{T/N}^T}\\
\rac B(C_{T/N})\ar[r]^{\alpha^{T/N}}&\rac B(T/N).
}
\]
Since the top diagram commutes, if we show that the exterior diagram commutes, then we will obtain
\begin{displaymath}
\bdef_{T/N}^T\alpha^T\bres_{C_T}^C=\alpha^{T/N}\bdef_{T/N}^T\bres_{C_T}^C.
\end{displaymath}
But by the Mackey formula we have $\bres_{C_T}^C\bind_{C_T}^C=\frac{|G|}{|T|}\textrm{Id}_{C_T}^{C_T}$, so $\bres_{C_T}^C$ is surjective on $\rac B(C_T)$,  hence we will have the result. 

By the relations appearing in Section 1.1 of \cite{boucbook}, we have 
\begin{displaymath}
\bres_{T/N}^{G/N}\bdef_{G/N}^{\,G}=\bdef_{T/N}^{\,T}\bres_T^G.
\end{displaymath}
So, since we are assuming that the FW morphism commutes with $\bdef_{G/N}^{\,G}$, and it commutes with restriction, we have that the exterior diagram commutes.
\end{proof}

Now let $e_D^C$ be one of the primitive idempotents in $\mathbb{Q}B(C)$. Since $\alpha^G$ is a ring homomorphism, $\alpha^G(e_D^C)$ must be a sum of primitive idempotents in $\mathbb{Q}B(G)$. To figure out which primitive idempotents appear in $\alpha^G(e_D^C)$, observe that if $K$ is a subgroup of $G$, then $|\alpha^G(e_D^C)^K|=|(e_D^C)^{C_K}|$, which is equal to 1 if $C_K=D$ and 0 otherwise, thus
\begin{displaymath}
\alpha^G(e_D^C)=\sum_{\substack{H\in [S_G]\\|H|=|D|}}e_H^G.
\end{displaymath}

Notice that if $G$ has no subgroups of order $|D|$, then $\alpha^G(e_D^C)=0$.

By Lemma \ref{ladef}, we have
\begin{equation}
\label{der}
\bdef_{G/N}^{\,G}\alpha^G(e_D^C)=\sum_{\substack{H\in [S_G]\\|H|=|D|}}t_{H,\, N}e_{HN/N}^{G/N},
\end{equation}
where 
\begin{displaymath}
t_{H,\, N}=\frac{|N_G(HN)/HN|}{|N_G(H)/H|}m_{H,\, H\cap N}.
\end{displaymath}

In a similar way  we obtain 
\begin{equation}
\label{izq}
\alpha^{G/N}\bdef_{C/C_N}^{\,C}(e_D^C)=\sum_{\substack{K/N\in [S_{G/N}]\\|K/N|=|DC_N/C_N|}}r_{D,\,C_N}e_{K/N}^{G/N},
\end{equation}
where
\begin{displaymath}
r_{D,\,C_N}=\frac{|D|}{|DC_N|}m_{D,\, D\cap C_N}.
\end{displaymath}

Hence, $\bdef_{G/N}^{\,G}\alpha^G=\alpha^{G/N}\bdef_{C/C_N}^{\,C}$ if and only if equations 1 and 2 coincide for every $D\leqslant C$. 

\begin{coro}
\label{corodefsub}
Suppose that for $G$ and $N$ the FW morphism commutes with $\emph{Def}_{G/N}^{\,G}$. Then for every subgroup $T$ of $G$ containing $N$, we have $m_{T,\, N}=m_{C_T,\, C_N}$.
\end{coro}
\begin{proof}
Using the idempotent $e_C^C$ in equations 1 and 2, we see that if they coincide, we have $m_{G,\,N}=m_{C,\,C_N}$. The result follows from Proposition \ref{defsub}.
\end{proof}

With these results, we are ready to prove the following proposition.

\begin{prop}
\label{propdef}
If  for $G$ and $N$ the FW morphism commutes with $\emph{Def}_{G/N}^{\,G}$, then $G$ and $N$ satisfy the gcd property and $N$ is a cyclic central subgroup of $G$. 
\end{prop}
\begin{proof}
 By the previous corollary, if $T$ is a subgroup of $G$ containing $N$ we have $m_{T,\, N}=m_{C_T,\, C_N}$. 

Now consider a cyclic subgroup $L$ of $G$ and take $T=LN$. Since $m_{T,\, N}=m_{C_T,\, C_N}$, in particular $m_{T,\, N} \neq 0$, by $iii)$ of Lemma \ref{propiedadesm}. 
Hence, $m_{T,\, T}$, which is equal to $m_{T,\, N}m_{T/N,\, T/N}$ by $i)$ of Lemma \ref{propiedadesm},
is also different from zero, since $T/N$ is a cyclic group. So, $T$ is a cyclic group too, by $ii)$ of Lemma \ref{propiedadesm}. Given that $L$ and $N$ are both contained in the cyclic group $T$, we obtain that $N$ is a cyclic group and that $L$ and $N$ satisfy $(|L|,\, |N|)=|L\cap N|$, so the gcd property holds for $N$ by $iv)$ of Lemma \ref{propiedad}. This also shows that for any $x\in G$, the elements of $N$ commute with $x$ (because $\langle x\rangle N$ is cyclic), hence $N$ is a central subgroup of $G$. 
\end{proof}

\begin{rem}
It is easy to prove, under the conditions of  the previous proposition, that $N$ must be contained in the intersection of the maximal cyclic subgroups of $G$, call it $M$. It is a question of Serge Bouc if $M$  is the maximal subgroup of $G$ satisfying the commutativity of the FW morphism with deflation. 
Unfortunately, proving that $M$ satisfies this commutativity seems to be too complicated.
\end{rem}

We finish the section by showing that if $G$ is a finite group containing a unique subgroup $N$ of order prime $p$, which is central, then it suffices to have $m_{T,\, N}=m_{C_T,\, C_N}$, for any subgroup $T$ of $G$ containing $N$, to conclude that the FW morphism commutes with $\bdef^{\,G}_{G/N}$.

\begin{lema}
Let $N\leqslant G$ be groups satisfying the gcd property with $N$ cyclic, and let $d$ be a divisor of $|G|$. Then $G$ has a subgroup of order $d$ if and only if $G/N$ has a subgroup of order $d/(d, \, |N|)$.
\end{lema}
\begin{proof}
It is clear that if $H$ is a subgroup of $G$ of order $d$, then $HN/N$ is a subgroup of $G/N$ or order $d/(d,\, |N|)$.

Now suppose that $G/N$ has a subgroup $K/N$ of order $a=d/(d,\, |N|)$ and let $N_1$ be the subgroup of $N$ of order $(d,\, |N|)$. Then $N_1$ is normal in $G$ and $K/N_1$ contains $N/N_1$, which is a normal subgroup of order $b=|N|/(d,\, |N|)$. Now, the index of $N/N_1$ in $K/N_1$ is $a$ and we have $(a,\, b)=1$. Hence, by the Schur-Zassenhaus Theorem, $K/N_1$ contains a subgroup of order $a$, and $K$ has a subgroup of order $d$.
\end{proof}

\begin{lema}
\label{paraejem}
Let $G$ be a finite group such that for a prime $p$, $G$ has a unique subgroup $N$ of order $p$ that is central. Then if $m_{T,\, N}=m_{C_T,\, C_N}$ for any subgroup $T$ of $G$ containing $N$, we have that the FW morphism commutes with $\emph{Def}^{\,G}_{G/N}$. 
\end{lema}
\begin{proof}
We continue to use the notation of Proposition \ref{propdef}.

Let $D$ be a subgroup of $C$ of order $d$. By the previous lemma, if $G$ does not have a subgroup of order $d$, then $G/N$ does not have a subgroup of order $DC_N/C_N=d/(d,\, |N|)$. Then by equations \ref{der} and \ref{izq}, we have
\begin{displaymath}
\bdef_{G/N}^{\,G}\alpha^G(e_D^C)=0=\alpha^{G/N}\bdef_{C/C_N}^{\,C}(e_D^C).
\end{displaymath}
Hence, it remains to consider the idempotents $e_D^C$ such that $G$ has a subgroup of order $d$. Suppose first that $p$ does not divide $d$. In this case, for $H\leqslant G$ of order $d$ we have $m_{H,\, H\cap N}=1=m_{D,\, D\cap C_N}$. Also, as in the proof of the previous lemma, for any $K/N$ appearing in $[s_{G/N}]$ with $|K/N|=|DC_N/C_N|$, we have that $K=HN$ for $H$ a subgroup of order $d$. Finally, it is clear that if $H$ is in $[s_G]$, then we can chose $[s_{G/N}]$ containing $HN/N$.  This means that, in order to have equations 1 and 2 coincide, it suffices to have
$N_G(HN)=N_G(H)$ for any $H$ of order $d$, but this is clear.

Now suppose that $p$ divides $d$. In this case, any group $T$ of order $d$ contains $N$, and we easily have that if $T$ is in $[s_G]$ then we can chose $[s_{G/N}]$ containing $T/N$. Hence, in order to have equations 1 and 2 coincide, it suffices to have $m_{T,\, N}=m_{C_T,\, C_N}$.
\end{proof}

\begin{ejem}
The following two cases satisfy the conditions of  the previous lemma  with $m_{T,\, N}=1=m_{C_T,\, C_N}$ for every $T\leqslant G$ containing $N$. The reason is that in both cases, the group $N$ is contained in $\Phi (G)$, thus by Example 5.2.3 of \cite{boucbook}, $m_{T,\, N}=1$, and $m_{C_T,\, C_N}=1$ by Lemma \ref{propiedadesm}. 

Notice that one of the groups $G$ is solvable and the other one is not.
\begin{itemize}
\item[i)] $G=SL(2,\, 5)$ and $N=Z(G)\cong C_2$. 
\item[ii)] Let $G$ be the dicyclic group of order $4m$ and $N=Z(G)\cong C_2$. Recall that $G$ has a presentation
\begin{displaymath}
G=\langle a,\, b\mid a^{2m}=1,\, b^2=a^m,\, bab^{-1}=a^{-1}\rangle.
\end{displaymath}
and that $N=\langle a^m\rangle$ is the only subgroup of order $2$ of $G$. 
\end{itemize}
\end{ejem}


If $G$ and $N$ are as in Lemma \ref{paraejem}, then $N$ is a minimal abelian normal subgroup of any $T\leqslant G$ containing $N$.  Under these hypothesis, Proposition 5.6.4 of \cite{boucbook} provides a formula for $m_{T,\, N}$ that could help to determine when it coincides with $m_{C_T,\, C_N}$. This would certainly provide other examples for which the FW morphism commutes with deflation.

\section*{Acknowledgments}

Many thanks to Serge Bouc, for without his knowledge of \textit{les \`ems}, this note would not exist, and thanks to the referee for the useful suggestions.


\end{document}